\documentclass{article}

\usepackage{microtype}
\usepackage[utf8]{inputenc}
\usepackage{graphicx}
\usepackage{subfigure}
\usepackage{booktabs} 
\usepackage{todonotes}
\usepackage{pgfplots}
\usepackage{xcolor}
\pgfplotsset{compat=newest}
\usepackage{dashrule}
\usepackage{hyperref}

\usepackage[preprint]{icml2019}

\usepackage{amsmath,amssymb,amsfonts,amsthm,bm,mathtools,commath}

\newcommand{\E}[2]{\mathbb{E}_{#1}\left[{#2}\right]}
\DeclareMathOperator*{\argmin}{\arg\!\min}

\DeclareMathOperator{\off}{off}
\DeclareMathOperator{\Exp}{Exp}

\renewcommand\abs[1]{\ensuremath{\left\lvert#1\right\rvert}}

\newcommand{\bbR}{\mathbb{R}}
\renewcommand{\O}[1]{\mathcal{O}\left(#1\right)}
\newcommand{\smallO}[1]{$o$\left(#1\right)}

\renewcommand{\P}{\mathcal{P}}
\newcommand{\N}{\mathcal{N}}
\newcommand{\Skew}{\mathrm{Skew}}
\newcommand{\symnorm}[1]{\norm{#1}_{F,\mathrm{sym}}}

\newtheorem{theorem}{Theorem}
\newtheorem{lemma}{Lemma}
\newtheorem{proposition}{Proposition}

\icmltitlerunning{Approximating Orthogonal Matrices with Effective Givens Factorization}

\begin{document}

\twocolumn[
\icmltitle{Approximating Orthogonal Matrices with Effective Givens Factorization}

\icmlsetsymbol{equal}{*}

\begin{icmlauthorlist}
\icmlauthor{Thomas Frerix}{tum}
\icmlauthor{Joan Bruna}{nyu}
\end{icmlauthorlist}

\icmlaffiliation{tum}{Technical University of Munich}
\icmlaffiliation{nyu}{New York University}

\icmlcorrespondingauthor{Thomas Frerix}{thomas.frerix@tum.de}
\icmlcorrespondingauthor{Joan Bruna}{bruna@cims.nyu.edu}

\icmlkeywords{Optimization, Computational Complexity, Numerical Linear Algebra}

\vskip 0.3in
]

\printAffiliationsAndNotice{}  

\begin{abstract}
We analyze effective approximation of unitary matrices.
In our formulation, a unitary matrix is represented as a product of rotations in two-dimensional subspaces, so-called Givens rotations.
Instead of the quadratic dimension dependence when applying a dense matrix, applying such an approximation scales with the number factors, each of which can be implemented efficiently.
Consequently, in settings where an approximation is once computed and then applied many times, such a representation becomes advantageous.
Although effective Givens factorization is not possible for generic unitary operators, 
we show that minimizing a sparsity-inducing objective with a coordinate descent algorithm on the unitary group yields good factorizations for structured matrices. 
Canonical applications of such a setup are orthogonal basis transforms.
We demonstrate numerical results of approximating the graph Fourier transform, which is the matrix obtained when diagonalizing a graph Laplacian.
\end{abstract}

\section{Introduction}
Unitary operators are ubiquitous in many areas, from numerical linear algebra to quantum computing and cryptography.
Celebrated applications include the QR-decomposition and the diagonalization of symmetric matrices \cite{Golub2000}.
Without any assumptions on the structure of the matrix, applying a unitary transformation in  $d$ dimensions requires $\mathcal{O}(d^2)$ operations for the matrix-vector product.
In scenarios where a given unitary operator needs to be intensively applied many times, using approximations that trade-off accuracy with a better scaling behavior in the dimension is desirable.

In this paper, we develop a method to compute approximations of unitary matrices in the form of Givens factorization \cite{Givens1958}.
Givens rotations are localized in a two-dimensional subspace of predefined coordinates. 
Therefore, computations with Givens sequences scale with the number of factors and the computational cost for applying each factor can be kept low since efficient implementations are possible \cite{Golub2012}.
Our main motivation comes from the success story of the Fast Fourier transform (FFT) \cite{Cooley1965}, which brought down the computational cost of applying a Fourier transform to $\mathcal{O}(d\log(d))$ operations. 
This reduction led to a revolution in signal processing and was recognized by \citet{Sullivan2000} as one of the most important algorithms of the 20th century.
However, this speed-up relies on the fact that the classical Fourier transform is defined over a periodic grid, which provides many symmetries leveraged in the butterfly structure of the FFT.

These symmetries do not carry over to unstructured domains such as graphs and general unitary operators.
In fact, using simple covering bounds, we show that generic unitary matrices require $\mathcal{O}(d^2/ \log d)$ Givens factors to be effectively approximated. 
However, the question of approximating with fewer factors in the presence of structure remains open: given an element $U \in U(d)$, how to produce the best possible $N$-term sequence of Givens rotations $G_1 \dots G_N$ that minimizes $\norm{U - \prod_j G_j}$ ?

Due to the combinatorial nature of selecting Givens subspaces, this is an NP-hard optimization problem.
In this paper, we propose a relaxation based on sparsity-inducing norms over the unitary group. In essence, given a point $U \in U(d)$, we use the gradient flow of a potential function $f: U(d) \to \mathbb{R}$ to define a path that links $U$ to its nearest signed permutation matrix, the sparsest elements of the group and thus the global minimizers of $f$.
Then, our algorithm tries to approximately follow this path using coordinate descent with the Givens factors acting as generators of the group. 

We validate our algorithm on a family of structured orthogonal operators, constructed with a planted random sequence of $K$ Givens factors and demonstrate that effective approximation is possible in the regime $K=\mathcal{O}(d \log d)$. 
Finally, we apply our algorithm to approximate a graph Fourier transform (GFT), the orthogonal matrix obtained when diagonalizing a graph Laplacian.

For ease of exposition, we restrict our discussion to approximating orthogonal group elements.
However, this does not impose a restriction on the outlined approaches, as they equally apply to the complex unitary group as well as the real orthogonal group.

\section{Related Work}
\label{sec:related_work}
Givens rotations were introduced by \cite{Givens1958} to factorize the unitary matrix that transforms a square matrix into triangular form.
The elementary operation of rotating in a two-dimensional subspace led to numerous successful applications in numerical linear algebra \cite{Golub2012}, in particular, for eigenvalue problems \cite{Golub2000}.
In this context, a Givens sequence factorizes a unitary basis transform, which is an operation of paramount importance to signal processing.

In contrast to signal processing on a Euclidean domain, recently there has been increased interest in signal processing on irregular domains such as graphs \cite{Shuman2013, Bronstein2017}.
In this setting, \citet{Magoarou2016} considered a truncated version of the classical Jacobi algorithm \cite{Jacobi1846} to approximate the orthogonal matrix that diagonalizes a graph Laplacian. Other notable strategies to efficiently approximate large matrices with presumed structure include multiresolution analysis \cite{kondor2014multiresolution} and sparsity \cite{kyng2016approximate}.

In quantum computation, approximate representation of unitary operators is a fundamental problem.
Here, a unitary operation that performs a computation on a quantum state needs to be represented by or approximated with few elementary single- and two-qubit gates, ideally polynomial in the number of qubits.
In the literature of quantum computing, a Givens rotation is commonly referred to as a two-level unitary matrix; a generic $n$-qubit unitary operator can be factorized in such two-level matrices with $\O{4^n}$ elementary quantum gates \cite{Vartiainen2004}.

An alternative viewpoint on Givens sequences was analyzed by \citet{Shalit2014}.
The authors considered manifold coordinate descent over the orthogonal group as sequentially applying Givens factors.
Consequently, the minimizing sequence of this algorithm yields a Givens factorization of the initial orthogonal matrix.

In this work, we analyze information theoretic properties of approximating unitary matrices via Givens factorization.
We then propose to minimize a sparsity-inducing objective via manifold coordinate descent in a regime where effective approximation is possible.
Subsequently, we apply this approach to approximate the graph Fourier transform and demonstrate that the proposed method can find better sequences compared to a truncated Jacobi algorithm.
This allows to efficiently transform a graph signal into the graph's approximate Fourier basis, an essential operation in graph signal processing.

\section{Givens Factorization and Elimination}
\label{sec:givens_factorization_intro}

Givens matrices represent rotations in a two-dimensional subspace, while leaving all other dimensions invariant \cite{Givens1958,Golub2012}.
Such a counter-clockwise rotation in the $(i,j)$-plane by an angle $\alpha$ can be written as applying $G^T(i, j, \alpha)$, where 
\begin{align}
	G(i, j, \alpha) = 
       \begin{bsmallmatrix}   1   & \cdots &    0   & \cdots &    0   & \cdots &    0   \\
                      \vdots & \ddots & \vdots &        & \vdots &        & \vdots \\
                         0   & \cdots &    \cos(\alpha)   & \cdots &    \sin(\alpha)   & \cdots &    0   \\
                      \vdots &        & \vdots & \ddots & \vdots &        & \vdots \\
                         0   & \cdots &   -\sin(\alpha)   & \cdots &    \cos(\alpha)   & \cdots &    0   \\
                      \vdots &        & \vdots &        & \vdots & \ddots & \vdots \\
                         0   & \cdots &    0   & \cdots &    0   & \cdots &    1
       \end{bsmallmatrix}
\end{align}
The trigonometric expressions appear in  the $i$-th and $j$-th rows and columns.
Any orthogonal matrix $U \in \bbR^{d \times d}$ that is a rotation, $U \in SO(d)$, can be decomposed into a product of at most $d(d-1)/2$ Givens rotations.
In general, there exist many possible factorizations.
If $U \in O(d)\setminus SO(d)$, then it cannot be represented directly by a sequence of Givens rotations.
However, a factorization can be obtained up to permutation with a negative sign, e.g., by flipping two columns.

In numerical linear algebra, Givens factors are often used to selectively introduce zero matrix entries by controlling the rotation angle.
This leads to a constructive factorization algorithm, which demonstrates a $d(d-1)/2$-factorization.
To this end, we start with the matrix $U \in SO(d)$ and introduce zeros on the lower diagonal column-wise from left to right and bottom to top within every column.
This is achieved by choosing the rotation subspace $(i,j)$ and a suitable rotation angle to zero-out the matrix element $(i,j)$.
The elimination order is illustrated for $d=4$ by
\begin{align}
\label{eq:givens_qr_elimination_sequence}
    \begin{pmatrix}
        * & * & * & * & \\
        3 & * & * & * & \\
        2 & 5 & * & * & \\
        1 & 4 & 6 & * &
    \end{pmatrix}
\end{align}
After $N=d(d-1)/2$ steps, we have $G_N^T\dots G_1^T U = D$, where $D$ is a diagonal matrix with $D_{kk} = -1$ for an even number of values and $D_{kk} = 1$ otherwise.
This result can be reduced to the identity by selecting two subspaces with values $D_{ii} = D_{jj} = -1$ and applying a rotation by an angle $\alpha = \pi$.
We refer to this algorithm by \textit{structured elimination}.

Apart from this sign ambiguity, we consider factorizations in the broader sense up to signed  permutation of the resulting matrix columns.
To be explicit, the set of signed permutation matrices is defined as $\P_d \coloneqq \{P \in \bbR^{d\times d} \vert P_{ij} \in \{-1,0,1\}, \sum_{i} \abs{P_{ij}} = 1 \;\forall j, \sum_{j} \abs{P_{ij}} = 1 \;\forall i\}$.
For a matrix $U\in O(d)$, to measure approximation quality, we denote an approximation by $\hat U$ and use a symmetrized Frobenius norm criterion up to a signed permutation matrix as follows:
\begin{align}
\label{eq:symnorm}
    \symnorm{U - \hat U} \coloneqq \min_{P\in \P_d} \norm{U - \hat U P}_F \;.
\end{align}
The range of \eqref{eq:symnorm} over the orthogonal group is $[0,\sqrt{2 d})$ as the maximum is obtained for the distance between Hadamard\footnote{A Hadamard matrix is an orthogonal matrix $H $whose entries satisfy $|H_{i,j}|=1/\sqrt{d}$ for all $i,j$.} matrices $H(d)$ and the identity with $\symnorm{H(d) - I}/\sqrt{d} \to \sqrt{2}$ as $d\to \infty$.
Since
    $\norm{A}_F^2 = \E{x\sim \N(0,I)}{\norm{Ax}_2^2} \;$,
the criterion measures the average approximation quality over random Gaussian vectors when applying $\hat U$ instead of $U$.
The motivation for this definition is twofold.
First, this definition allows us to discuss Givens factorizations of orthogonal matrices with negative determinant and henceforth we consider factorization over the orthogonal group $O(d)$ rather than the special orthogonal group $SO(d)$.
Second, it enlarges the class of possible factorization algorithms to those that cannot distinguish between signed permutation matrices.
Observe that since the cost of multiplying by a signed permutation matrix is $\O{d}$ \cite{Knuth1998}, the computational efficiency arguments in this paper are not affected by the permutation equivalence class as we are discussing approximations in the regime of $\O{d\log(d)}$ factors.

\section{Information Theoretic Rate of Givens Representation}
\label{sec:orthogonal_group_approximation}
The elimination algorithm discussed in Section~\ref{sec:givens_factorization_intro} guarantees to factorize any orthogonal matrix in at most $d(d-1)/2$ Givens factors, which corresponds to the dimension of the orthogonal group. Since each Givens factor is parametrized by a single angle, it immediately follows that exact Givens factorization for arbitrary elements $U \in O(d)$ necessarily requires $d(d-1)/2$ factors. 

Hence, this leads to the question of approximate factorization: if one tolerates a certain error $\|U - \hat{U}\|_F \leq \epsilon$, is it possible to find approximations $\hat{U}=\prod_{n \leq N} G_n$ with $N=o(d^2)$, ideally with $N=\mathcal{O}(d \log d)$?
A covering argument shows that generic orthogonal matrices in $d$ dimensions require at least $\Theta(d^2 / \log(d))$ Givens factors to achieve an $\epsilon$-approximate factorization. We denote by $\mu$ the uniform Haar measure on the unitary group, which we normalize for each $d$, $\mu(U(d))=1$.
For notational simplicity, we carry out the proof for the operator $2$-norm.
An analogous argument holds by replacing the operator $2$-norm with the Frobenius norm while re-scaling the error by $\sqrt{d}$.
\begin{lemma}
\label{thm:perturbed_factor_product}
Let $\prod_{n \leq N} G_n$ be a product of Givens factors with rotation angles $\alpha_n$ and $\bar{G}_n$ be the respective perturbed factors with rotation angles $\alpha_n + \delta_n$ and perturbations $0 \leq \delta_n \leq \delta$. Then,
\begin{align}
\label{eq:perturbed_factor_product}
    \norm{\prod_{n \leq N} \bar{G}_n - \prod_{n \leq N} G_n}_F \leq 2 N \delta \;.
\end{align}
\end{lemma}
\begin{proof}
For any orthogonal matrices $U,U',V,V'$, we have
\begin{align}
    \norm{U'V' - UV}_F &= \norm{(U + U' - U) V' - UV}_F \nonumber\\
                       &\leq \norm{U (V' - V) }_F + \norm{(U' - U) V'}_F \nonumber \\
                       &= \norm{V' - V}_F + \norm{U' - U}_F,
\end{align}
by using the fact that the Frobenius norm is invariant to orthogonal matrix multiplication.
By iterating this relation, we obtain
\begin{align}
    \norm{\prod_{n \leq N} \bar{G}_n - \prod_{n \leq N} G_n}_F \leq \sum_{n\leq N} \norm{\bar{G}_n - G_n}_F \;.
\end{align}
Since $\bar{G}_n$ and $G_n$ rotate in the same subspace,
\begin{align}
    \norm{\bar{G}_n - G_n}_F = 2 \sqrt{1 - \cos(\delta_n)} \;.
\end{align}
Inequality \eqref{eq:perturbed_factor_product} follows from $\sqrt{1 - \cos(\delta_n)} \leq \delta_n \leq \delta$.
\end{proof}
\begin{theorem}
\label{thm:covering}
Let $\epsilon>0$. If $N = \smallO{d^2 / \log(d)}$, then as $d\to \infty$, 
$$\mu \left(\left\{U \in U(d) \biggr\vert \inf_{G_1\dots G_N} \| U - \prod_n G_n \|_2 \leq \epsilon \right\}\right) \to 0~.$$
\end{theorem}
\begin{proof}
Consider an $\epsilon$-covering of the unitary group, i.e., a discrete set $\mathcal{X}$ such that $\inf_{X \in \mathcal{X}} \| U - X\|_2 \leq \epsilon$ for all $U \in U(d)$. 
Since the manifold dimension of the unitary group is $d(d-1)/2$, we need $|\mathcal{X}|= \Theta({\epsilon^{-d(d-1)/2}})$ many balls for that cover.
Let $N \coloneqq N(d)$ be the number of available Givens factors for approximation at dimension $d$, and $\mathcal{A}_{N} = \{ X \in U(d) \vert \inf_{G_1\dots G_{N}} \| X - \prod_{n\leq N} G_n \|_2 \leq \epsilon/2\}$ denote the set of unitary operators which can be effectively approximated with $N$ Givens terms. 
Now, suppose that $\mu(\mathcal{A}_{N}) \geq c>0$, i.e., the set of group elements admitting an $\epsilon/2$-approximation has positive measure. This implies that any $\epsilon$-cover of $\mathcal{A}_{N}$ must be of size $\Theta({\epsilon^{-d(d-1)/2}})$. Let us build such an $\epsilon$-cover. 

If we discretize the rotation angle to a value $\delta > 0$, then there are $(d(d-1)/2\delta)$ many different quantized Givens factors, denoted by $\bar{G}_i$, and consequently ${(d(d-1)/2\delta)^N}$ many different sequences.
It follows that if $\delta \coloneqq \frac{\epsilon}{4N}$, the discrete set $\mathcal{Y}=\{\prod_{n \leq N} \bar{G}_{i_n}\}$ containing all possible sequences of length $N$ of quantized Givens rotations is an $\epsilon$-cover of $\mathcal{A}_N$. 
Indeed, by using Lemma~\ref{thm:perturbed_factor_product} and the fact that the operator $2$-norm is bounded by the Frobenius norm,  we have $\forall\,X \in \mathcal{A}_N\,,$
$$\| X - \prod_{n \leq N} \bar{G}_n\|_2 \leq \| X - \prod_{n\leq N} G_n \|_2 + 2 N \delta \leq \frac{\epsilon}{2} + \frac{\epsilon}{2}=\epsilon~. $$
Since $|\mathcal{Y}| =\left(\frac{2d(d-1)N}{\epsilon}\right)^N$, it follows that 
$$\left(\frac{2d(d-1)N}{\epsilon}\right)^N = \Theta({\epsilon^{-d(d-1)/2}})~,$$
which implies $N = \O{d^2 / \log d}$.
\end{proof}

An immediate consequence of Theorem~\ref{thm:covering} is that generic effective approximation, i.e., with a number of factors $N=\O{d \log d}$, is information theoretically impossible. 
However, the situation may be entirely different for structured distributions of unitary operators.
For that purpose, we develop an algorithm to obtain effective approximations based on sparsity-inducing norms. 

\section{Givens Factorization and Coordinate Descent on $O(d)$}
\label{sec:coordinate_descent}
In this section, we offer an alternative viewpoint presented by \citet{Shalit2014} that interprets Givens factorization as manifold coordinate descent on the orthogonal group over a certain potential energy. 

The orthogonal group $O(d)$ is a matrix Lie group  with associated Lie algebra $\mathfrak{o}(d) = \Skew(d) = \{X \in \bbR^{d\times d} \vert X = -X^T\}$, the set of $d\times d$ skew-symmetric matrices \cite{Hall2003}.
The tangent space at an element $U$ is $T_U O(d) = \{XU \vert X \in \Skew(d)\}$ and the Riemannian directional derivative of a differentiable function $f$ in the direction $XU\in T_U O(d)$ is given by
\begin{align}
    D_X f(U) = \eval{\od{}{\alpha}f(\Exp(\alpha X)U)}_{\alpha=0} \;,
\end{align}
where $\Exp : \mathfrak{o}(d) \rightarrow O(d)$ is the matrix exponential.
If we choose the basis $\{ X_{ij} = e_i e_j^T - e_j e_i^T \vert 1 \leq i \leq j \leq d\}$ for the tangent space, then $D_{X_{ij}} f(U)$ represents the directional derivative in such a coordinate direction.
A coordinate descent algorithm uses a criterion to choose coordinates $(i,j)$ and a step size (rotation angle) $\alpha$ to iteratively update
\begin{align}
\label{eq:coord_descent_iteration}
    U^{k+1} = \Exp(-\alpha X_{ij}) U^k \;.
\end{align}
A greedy criterion determines the best descent on $f$ by a search over all possible coordinate directions $\{X_{ij}\}_{i\leq j \leq d}$ with the optimal step size obtained by a line search.

A Givens factor can be interpreted as a coordinate descent step over the orthogonal group.
This follows from the relation
\begin{align}
    \Exp(-\alpha X_{ij}) = G^T(i,j,\alpha) \;.
\end{align}
In $d=3$, an explicit example of the correspondence between Lie algebra and Lie group elements is
\begin{align} 
\left(
\begin{array}{ccc}
 0 & 0 & 0 \\
 0 & 0 & -\alpha \\
 0 & \alpha & 0\\
\end{array}
\right) 
\longrightarrow
\left(
\begin{array}{ccc}
 1 & 0 & 0 \\
 0 & \cos (\alpha ) & -\sin (\alpha ) \\
 0 & \sin (\alpha ) & \cos (\alpha ) \\
\end{array}
\right) ~.
\end{align}
Suppose we want to minimize a function $f$ over the orthogonal group, 
\begin{align}
\label{eq:generic_manifold_opt}
    \min_{U\in O(d)} f(U)  \;.
\end{align}
Then minimizing \eqref{eq:generic_manifold_opt} with manifold coordinate descent iterations \eqref{eq:coord_descent_iteration} yields a Givens factorization of the initial point $U^0$.
A truncated sequence leads to an approximate factorization.
From this viewpoint, the quality of a Givens factorization can be controlled by properties of the function $f$.
In the following, we construct an objective function that results in approximate factorization with less than $\mathcal{O}(d^2)$ factors.

\section{Sparsity-Inducing Dynamics}
\label{sec:l1_criterion}
To factorize a matrix $U \in O(d)$ one may choose it as an initial value to problem \eqref{eq:generic_manifold_opt} when minimizing a suitable potential function $f$ with manifold coordinate descent.
We want to find a factorization up to signed permutation of the matrix columns.
As the signed permutation matrices are the sparsest orthogonal matrices, we consider an energy function that quickly enforces sparsity, the element-wise $L_1$-norm of a matrix,
\begin{align}
\label{eq:matrix_l1}
    f(U) \coloneqq d^{-1}\norm{U}_1 = d^{-1} \sum_{i,j=1}^d \abs{U_{ij}} \;.
\end{align}
Although $f$ is convex in $\mathbb{R}^{d^2}$ (since it is a norm), due to the non-convexity of the domain, the problem $\min_{U \in O(d)} f(U)$ is non-convex . The landscape of $f$ characterizes the class of orthogonal matrices that admit effective Givens approximation. 
It is easy to see that the global minima of $f$ in $O(d)$ consist of signed permutation matrices, with $\min f(U)=1$, and the global maxima are located at Hadamard matrices, with $\max f(U) = \sqrt{d}$. 
A more involved question concerning the presence or absence of spurious local minima of $f$ is of interest. 
The following proposition partially addresses this question by showing that critical points of $f$ are necessarily located at $U \in O(d)$ with some of its entries set to zero. 
\begin{proposition}
\label{thm:l1_global_angle_opt}
Let $x\in \bbR^{2\times d}$ and let
\begin{align}
R(\alpha) \coloneqq \
\begin{bmatrix}
    \cos(\alpha) & -\sin(\alpha) \\
    \sin(\alpha) & \cos(\alpha)
\end{bmatrix}
\end{align}
be a counter-clockwise rotation in the plane by an angle $\alpha$. 
Consider the function $g(\alpha) \coloneqq \norm{R(\alpha) x}_1$.
Then, at every local minimum $\alpha^*$ of $g$ there exist indices $k,l$ such that $\left(R(\alpha^*) x\right)_{kl} = 0$.
\end{proposition}
\begin{proof}
We show equivalently that any stationary point $\alpha^*$ with $\left(R(\alpha^*) x\right)_{kl} \neq 0 \; \forall k,l$ is a local maximum.
At any such point the function $g$ is twice continuously differentiable and the second derivative is
\begin{align}
    \frac{\partial^2 g}{\partial \alpha^2}\biggr\vert_{\alpha = \alpha^*} = - g(\alpha^*) < 0 \;.
\end{align}
Consequently, any stationary point under this assumption must be a local maximum.
\end{proof}
Proposition~\ref{thm:l1_global_angle_opt} implies that for a given subspace $(i,j)$, the best rotation angle can be found by checking all axis transitions for the 2D points $(u_{ik}, u_{jk}), k\in \{1,\dots d\}$ and selecting the angle that most minimizes the objective among them.
It also implies that any local minimum of $f$ must correspond to an orthogonal matrix with at least $d$ zeros placed at specific entries, such that no two rows or columns have the same support. 
Indeed, Proposition \ref{thm:l1_global_angle_opt} implies that there exists a continuous path $t\mapsto U(t)=G(i,j,\alpha(t))$ with $\alpha(0)=0$, generated by a Givens rotation of angle $\alpha(t)$, such that $f(U(t))$ is non-increasing at $t=0$, provided one can find two rows or columns of $U$ with the same support. However, this result does not exclude the possibility that $f$ has spurious local minima at matrices $U$ with the above special sparsity pattern. In fact, we conjecture that the landscape of $f$ does have spurious local minima.

A manifold coordinate descent on the objective function $f$ is explicitly stated in Algorithm~\ref{alg:l1_coordinate_descent}.
The crucial step involves optimizing this objective in the rotation angle $\alpha$ for a given subspace $(i,j)$, which is a non-convex optimization problem.
Nevertheless, the global optimum can be found as stated by Proposition~\ref{thm:l1_global_angle_opt}.
In $d$ dimensions, this step requires $d$ operations.
Consequently, due to the squared dimension dependence of the double for-loop, a naive implementation of Algorithm~\ref{alg:l1_coordinate_descent} would require $\O{d^3}$ operations .
However, applying the selected Givens factor in each step changes only two rows of the matrix; thus, in the subsequent iteration, only those pairs of rows that involve the previously modified ones need to be re-computed.
These are $\O{d}$ rows and altogether the runtime of an iteration is $\O{d^2}$.

\begin{algorithm}[tb]
   \caption{Coordinate descent on the $L_1$-criterion}
   \label{alg:l1_coordinate_descent}
\begin{algorithmic}
   \STATE {\bfseries Input:} initial value $U^0 \in O(d), f(U) = \norm{U}_1$
   \REPEAT
       \FOR{$i=1$ {\bfseries to} $d$}
           \FOR{$j=1$ {\bfseries to} $d$}
                \IF {$\alpha_{ij}^*$ not up-to-date}
                    \STATE $\alpha_{ij}^* = \argmin_{\alpha} f(G^T(i,j,\alpha)U^k)$
                \ENDIF
           \ENDFOR
       \ENDFOR
       \STATE $i^*, j^* = \argmin_{i,j} f(G^T(i,j,\alpha_{ij}^*)U^k)$
       \STATE $U^{k+1} = G^T(i^*,j^*,\alpha_{i^*j^*}^*)U^k$
   \UNTIL{$\symnorm{U^{k+1} - I} < \varepsilon$ or $maxIter$ is reached}
\end{algorithmic}
\end{algorithm}

\section{Numerical Experiments}
\label{sec:factorization_results}

\subsection{Planted Models}
\label{sec:planted_models}
Theorem~\ref{thm:covering} shows that we cannot expect to find good approximations to Haar-sampled matrices with less than $\O{d^2/\log(d)}$ Givens factors.
Therefore, we focus on a distribution for which we can control approximability.
We use the uniform distribution over the set $\{U \in SO(d) \vert U = G_1\cdots G_K, G_k = G(i_k, j_k, \alpha_k)\}$, where each $G_k$ is obtained by first sampling a subspace uniformly at random (with replacement), and then sampling the corresponding angle uniformly from $(0,2\pi)$.
We denote the resulting distribution by the $K$-planted distribution $\nu_K$.
While this distribution may be sparse in the number of Givens factors for $K \ll d(d-1)/2$, this does not imply that the resulting matrices are sparse. 
In fact, products of Givens matrices become dense quickly. 
It follows from the Coupon Collector's Lemma that matrices generated with  $\Theta(d \log_2(d))$ Givens factors are already dense with high probability. 
To visualize this effect, Figure~\ref{fig:planted_solution_sparsity} shows the $L_0$-norm as a function of planted Givens factors.
\begin{figure}[t]
	\centering
     \resizebox{\columnwidth}{!}{\input{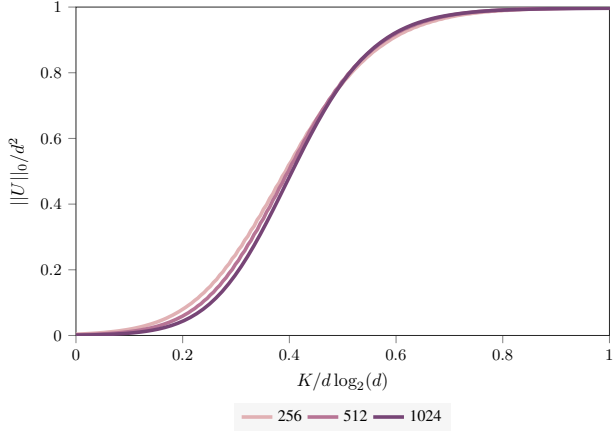}}
    \caption{Average sparsity based on $100$ samples of matrices drawn from the $K$-planted distribution over $SO(d)$ for increasing $K$. Standard deviation is negligible and not shown. Matrices become dense quickly as the number of planted Givens factors grows. In particular, matrices sampled from the $d\log_2(d)$-planted distribution are already dense.}
    \label{fig:planted_solution_sparsity}
\end{figure}

We compare the following factorization algorithms.
A \textit{greedy baseline} iteratively finds the Givens factor that most minimizes the objective \eqref{eq:symnorm}.
The \textit{structured elimination} algorithm described in Section~\ref{sec:givens_factorization_intro} yields a sequence of Givens factors that eliminate matrix entries in the order \eqref{eq:givens_qr_elimination_sequence} and is guaranteed to find a perfect factorization with $d(d-1)/2$ factors.
Our sparsity-inducing algorithm minimizes the $L_1$-criterion \eqref{eq:matrix_l1} via a \textit{manifold coordinate descent} scheme.\footnote{An implementation of these algorithms can be found at\\ \url{https://github.com/tfrerix/givens-factorization}}

In an initial experiment, we demonstrate the approximation effectiveness of these algorithms; the results are shown in Figure~\ref{fig:dlogd_planted_approximation}.
They indicate that minimizing the $L_1$-criterion improves over directly minimizing the Frobenius norm (greedy baseline).
\begin{figure}[t]
	\centering
	\vspace{0.1cm}
    \resizebox{\columnwidth}{!}{\input{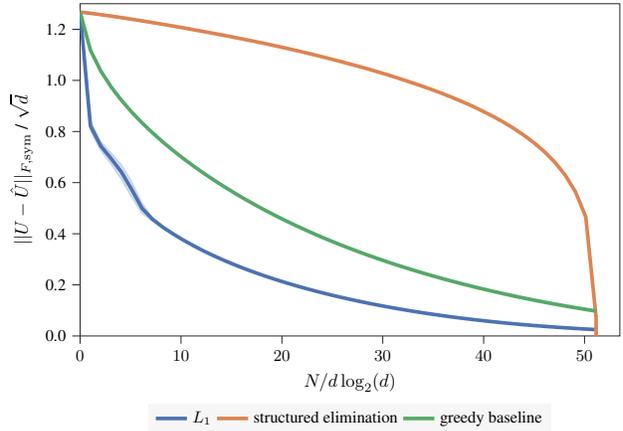}}
    \caption{Average Frobenius norm approximation error in $d~=~1024$ dimensions when factorizing $10$ samples drawn from the $d\log_2(d)$-planted distribution over $SO(d)$ with $d(d-1)/2$ factors. Shaded area denotes standard deviation.}
    \label{fig:dlogd_planted_approximation}
\end{figure}
Next, we analyze the approximability of samples drawn from the $K$-planted distribution $\nu_K$ as a function of $K$.
To obtain a Givens sequence, we factorize these samples with manifold coordinate descent on the $L_1$-objective \eqref{eq:matrix_l1}.
Along the optimization path, we define $N_\epsilon(U)$ as the number of Givens factors for which the normalized approximation error \eqref{eq:symnorm} is smaller than $\epsilon = 0.1$, i.e., 
\begin{align}
\label{eq:givens_factors_at_error_threshold}
    N_\epsilon(U) \coloneqq \min \left\{N \biggr\vert \frac{\symnorm{U - G_1\dots G_N}}{\sqrt{d}} < \epsilon\right\} 
\end{align}
We refer to a Givens sequence with such $N_\epsilon(U)$ factors as an $\epsilon$-factorizing sequence of $U$.
In Figure~\ref{fig:k_planted_approximation}, the sample average $N_\epsilon = n^{-1} \sum_{i=1}^n N_\epsilon(U_i)$ for $n=10$ samples is shown as a function of $K$.
We are interested in the rate at which $N_\epsilon$ grows for increasing $K$.
The data in Figure~\ref{fig:k_planted_approximation} show that for $K=\alpha d\log_2(d)$ and $N_\epsilon=\beta d\log_2(d)$, the ratio $\beta/\alpha$ is not independent of $d$.
For the shown dimension regime this implies that for $K=\O{d\log(d)}$, $N_\epsilon$ grows polynomial in $d$, albeit with small rate for few planted factors.
To make this relation more precise, we extract the exponent $\eta$ of a model $N_\epsilon \sim d^\eta$.
Figure~\ref{fig:k_planted_rate} shows that the growth is slightly superlinear in the few-factor regime and becomes quadratic towards $K=d\log_2(d)$. 
Analytically characterizing such growth is left for future work. 

That said, our initial results suggest the existence of a computational-to-statistical gap for the recovery (or detection) of sparse planted Givens factors. 
Indeed, Theorem 1 proves that recovery with $K~=~\O{d^2 /\log d}$ planted factors is information-theoretically possible, whereas our greedy recovery strategy is only effective for $K~=~\O{d \log d}$. 
The mathematical analysis of our coordinate descent algorithm in the regime where effective approximation is feasible is beyond the scope of the present paper.
In particular, proving that $N_\epsilon = \O{d \log d}$  is sufficient when $K \lesssim d\log d$ remains an open question. 

\begin{figure}[t]
	\centering
    \resizebox{\columnwidth}{!}{\input{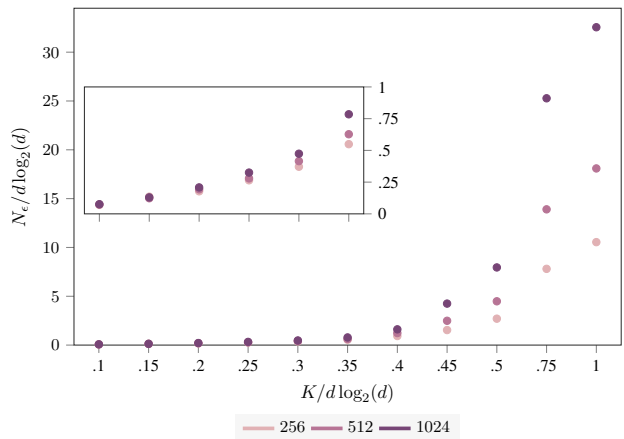}}
    \caption{Average number of Givens factors necessary to factorize a $K$-planted matrix in $d\in \{256, 512, 1024\}$ dimensions up to desired accuracy as a function of $K$. Here, $\epsilon = 0.1$ is the accuracy as defined in expression \eqref{eq:givens_factors_at_error_threshold}. Note that the x-axis is shown with unequal spacing to highlight the relevant regime of the data. The inset plot shows a zoom of the first data points.}
    \label{fig:k_planted_approximation}
\end{figure}

\begin{figure}[t]
	\centering
    \resizebox{\columnwidth}{!}{\input{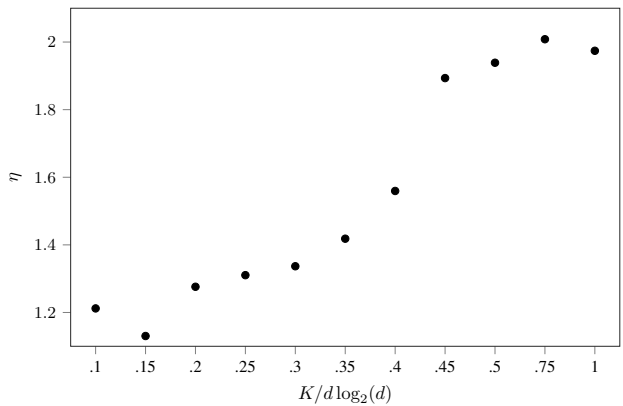}}
    \caption{Polynomial growth rate $\eta$ of the model $N_\epsilon \sim d^{\eta}$ as a function of the number of planted factors estimated from $d\in \{256, 512, 1024\}$. Note that the x-axis is shown with unequal spacing to highlight the relevant regime of the data.}
    \label{fig:k_planted_rate}
\end{figure}

\subsection{Application: Graph Fourier Transform}
\label{sec:gft_application}
The method introduced in this paper is useful in situations where one at first computes an approximation to a unitary operator, which is subsequently applied many times.
Hence, the trade-off between initial computation and approximation on the one hand and efficient application on the other hand is in favor of the latter.
Canonical examples for this scenario are orthogonal basis transforms.
In this paper, we draw motivation from the FFT, which yields a speed-up of applying a Fourier transformation over a regular grid domain from $\mathcal{O}(d^2)$ to $\mathcal{O}(d\log(d))$ time complexity \cite{Cooley1965}.
However, these speed-ups do not carry over when the domain is unstructured, such as general graphs.
Here, we compute an effective approximation of the graph Fourier transformation (GFT).
Consider a simple, undirected graph with degree matrix $D$ and adjacency matrix $A$.
The unnormalized graph Laplacian is defined as $L\coloneqq D - A$, which is a positive semi-definite, symmetric matrix.
The GFT is represented by the orthogonal matrix that diagonalizes $L$. 

A baseline for our method is the Jacobi algorithm \cite{Jacobi1846}, which diagonalizes a symmetric matrix $L$ by greedily minimizing the off-diagonal squared Frobenius norm, 
\begin{align}
\label{eq:off_diagonal_fro}
    \off(L) \coloneqq  \norm{L}_F^2 - \sum_{k=1}^d L_{kk}^2 \;.
\end{align}
This is achieved by zeroing-out the largest matrix element in absolute value at every iteration.
To this end, a Givens matrix similarity transformation with a suitably chosen rotation subspace and rotation angle is applied.
However, the Jacobi algorithm does not guarantee factorization in a finite number of steps; in particular, it may take more than $N=d(d-1)/2$ iterations.
In fact, the algorithm converges linearly \cite{Golub2012},
\begin{align}
    \off(L^{k+1}) \leq \left( 1 - \frac{1}{N} \right) \off(L^k) \;.
\end{align}
If the iteration number $k$ is large enough, quadratic convergence was shown by \citet{Schoenhage1964}.
Hence, the method is ineffective for small iteration numbers and in high dimensions.
A truncated version of this algorithm was used by \citet{Magoarou2016} to obtain an approximation to the GFT.
The objective \eqref{eq:off_diagonal_fro} of the Jacobi method is motivated by approximating the spectrum of the symmetric matrix through the Gershgorin circle theorem \cite{Gershgorin1931}.
However, we argue here that a criterion focused on approximating the eigenbasis of the symmetric matrix directly yields a more effective approximation to this orthogonal basis transformation.
We consider the eigendecomposition $L=U\Lambda U^T$ and compute an approximation of the orthogonal matrix $U$ with the algorithms outlined in Section~\ref{sec:planted_models}.
We demonstrate this procedure on Barabási-Albert random graphs and several real world graphs.

The Barabási-Albert model starts with $n_0$ unconnected vertices and iteratively adds vertices to the graph, which are connected to a number $m$ of already existing ones with a probability proportional to the degree of these vertices.
This construction is known as preferential attachment and induces a scale-free degree distribution found in real world graphs \cite{Barabasi1999}.
The details of generating these graphs are described in Table~\ref{tab:barabasi_albert}.
\begin{table}[t]
\caption{Construction of Barabási-Albert graphs. An $n$-vertex graph is constructed by choosing $n_0=m_k$ initial vertices, then adding vertices and connecting them to $m_k$ of already existing ones with a probability proportional to the degree of these vertices. $m_k$ is chosen such that the number of resulting edges is approximatly $k\cdot 0.25n(n-1)/2$.}
\label{tab:barabasi_albert}
\vskip 0.15in
\begin{center}
\begin{small}
\begin{sc}
\begin{tabular}{lccccc}
\toprule
$n$ & $64$ & $128$ & $256$ & $512$ & $1024$ \\
\toprule
\toprule
$m_1$ & $54$ & $109$ & $218$ & $437$ & $874$ \\
\midrule
$m_2$ & $36$ & $69$ & $136$ & $267$ & $528$ \\
\bottomrule
\end{tabular}
\end{sc}
\end{small}
\end{center}
\vskip -0.1in
\end{table}
We approximate the corresponding graph Laplacians with $n\log_2(n)$ factors leading to the results shown in Figure~\ref{fig:barabasi_albert}.
While our sparsity-inducing algorithm yields better factorizations in most cases, there exist scenarios, where the greedy baseline results in better approximations ($d\in\{512, 2014\}$ for $\sim 0.25 n \log_2(n)$ edges).
\begin{figure}[t]
	\centering
    \resizebox{\columnwidth}{!}{\input{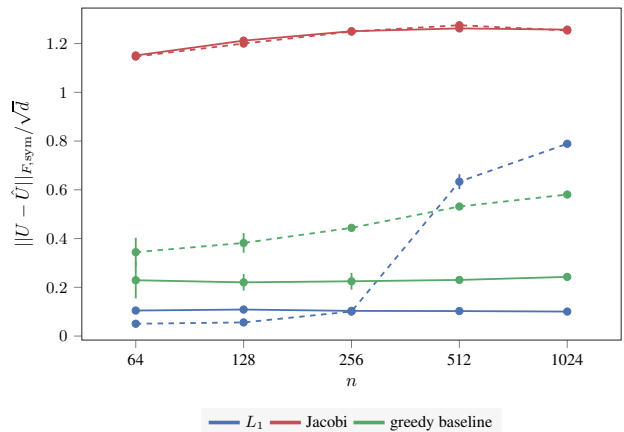}}
    \caption{Approximate factorization of the graph Laplacian of $n$-vertex Barabási-Albert graphs with $n\log_2(n)$ factors. Data points are averages of $10$ samples, vertical lines denote standard deviation. The solid (--) lines show factorizations of graphs with ${\sim 0.5 n(n-1)/2}$ edges, while the dashed (- -) lines show factorizations of graphs with $\sim 0.25 n(n-1)/2$ edges.}
    \label{fig:barabasi_albert}
\end{figure}
Finally, we demonstrate approximate factorization of the graph Laplacian of various real world graphs listed in Table~\ref{tab:real_world_graphs}.
Our $L_1$-algorithm yiels the best factorization for the Minnesota, HumanProtein, and EMail graphs, while the greedy baseline algorithm is superior for the Facebook graph.

A simple strategy to improve the performance of our $L_1$ greedy method with mild computational overhead is to perform beam-search, which is beyond the scope of this paper. 
Overall, it remains an open question to more closely characterize the graphs for which our sparsity-inducing algorithm yields effective approximations of the GFT.

\begin{table}[t]
\caption{GFT approximation for real world graphs with $n$ vertices and $n_e$ edges.}
\label{tab:real_world_graphs}
\vskip 0.15in
\begin{center}
\begin{small}
\begin{tabular}{lcc}
\toprule
& $n$ & $n_e$ \\
\toprule
\toprule
\textsc{Minnesota} & 2642 & 3304 \\
\cite{pygsp} &&\\
\midrule
\textsc{HumanProtein} & 3133 & 6726 \\
\cite{Rual2005} &&\\
\midrule
\textsc{EMail} & 1133 & 5451 \\
\cite{Guimera2003} &&\\
\midrule
\textsc{Facebook} & 2888 & 2981 \\
\cite{McAuley2012} &&\\
\bottomrule
\end{tabular}
\end{small}
\end{center}
\vskip -0.1in
\end{table}

\begin{figure}[t]
	\centering
    \resizebox{\columnwidth}{!}{\input{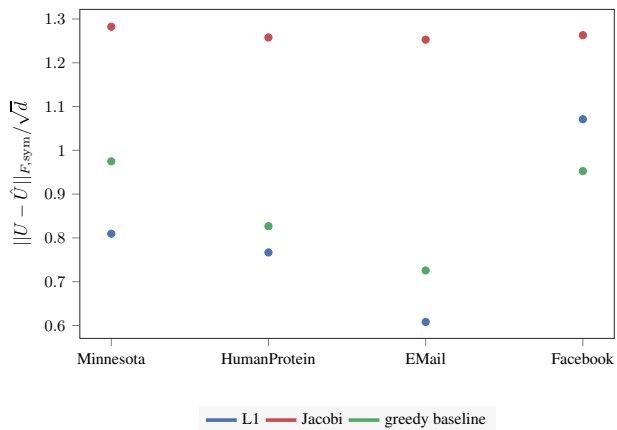}}
    \caption{Approximate factorization of the graph Laplacian of various $n$-vertex real world graphs with $n\log_2(n)$ factors.}
    \label{fig:real_world_graphs}
\end{figure}

\section{Discussion}
We analyzed the problem of approximating orthogonal matrices with few Givens factors.
While a perfect factorization in $\O{d^2}$ is always possible,  an approximation with fewer factors is advantageous if the orthogonal matrix is applied many times.
We showed that effective Givens factorization of generic orthogonal matrices is impossible and inspected a distribution of planted factors, which allows us to control approximability. 
Our initial results suggest that sparsity inducing factorization is promising beyond the sparse matrix regime.
However, it remains an open problem to further characterize the matrices that admit effective factorization using manifold coordinate descent on an $L_1$-criterion.\\
This work opens up questions we believe are important both from a theoretical and an applied perspective. 
On the theory side, important problems arising from our analysis are: \emph{(i)} a complete description of the landscape of $f(U)=\|U\|_1$ over the orthogonal and unitary groups, \emph{(ii)} a precise classification of the detection threshold $K(d)$ below which it is possible to discriminate a $K$-planted sample from a Haar sample in polynomial time, and \emph{(iii)} a guarantee that the proposed sparse Givens coordinate descent algorithm requires $N=\Theta(d \log d)$ terms for $K\leq C d \log d$ for some constant $C>0$. 
These questions suggest a learning approach whereby our sparsity promoting potential $f$ would be replaced by a classifier $f_\theta$ trained to discriminate between $K$-planted and Haar distributions. 
From an applied perspective, the method allows to approximately invert a time-varying symmetric linear operator $H(t)$. Similar to the Woodbury formula for low-rank updates of an inverse, one could set up an approximate Givens factorization of the eigenbasis of $H(t_0)$, and update it efficiently at subsequent times. If successful, this could dramatically improve the efficiency of second-order optimization schemes, where $H(t)$ is the Hessian of a loss function. 

\section*{Acknowledgements}
The authors would like to thank Oded Regev for in-depth discussions and early feedback, as well as Kyle Cranmer and Lenka Zbedorova for fruitful discussions on the topic, and Yann LeCun for introducing us to the problem. This work was partially supported by NSF grant RI-IIS 1816753, NSF CAREER CIF 1845360, the Alfred P. Sloan Fellowship and Samsung Electronics. 

\newpage
\bibliography{bibliography}
\bibliographystyle{icml2019}

\end{document}